\pdfoutput=1
\RequirePackage{ifpdf}
\ifpdf 
\documentclass[pdftex]{sigma}
\else
\documentclass{sigma}
\fi

\numberwithin{equation}{section}

\newtheorem{Theorem}{Theorem}[section]
\newtheorem*{Theorem*}{Theorem}

\newtheorem{Proposition}[Theorem]{Proposition}
\newtheorem*{rhproblem}{Riemann--Hilbert problem}
{ \theoremstyle{definition}

 \newtheorem{Remark}[Theorem]{Remark}
}

\usepackage[mathscr]{euscript}
\usepackage{dsfont}
\usepackage{tikz}

 \DeclareMathOperator{\Leb}{Leb}

 \newcommand*{\Cset}{\mathbb{C}} 
 \newcommand*{\Rset}{\mathbb{R}} 
 \newcommand*{\Nset}{\mathbb{N}} 

 \usetikzlibrary{decorations.pathreplacing}
 \usetikzlibrary{arrows,decorations.markings,intersections}
 \usetikzlibrary{shapes.misc}
 \usetikzlibrary{patterns}
 \tikzset{cross/.style={cross out, draw=black, minimum size=2*(#1-\pgflinewidth), inner sep=0pt, outer sep=0pt}, cross/.default={2pt}}
 \usetikzlibrary {patterns.meta}
 \tikzdeclarepattern{
 name=mylines,
 parameters={
 \pgfkeysvalueof{/pgf/pattern keys/size},
 \pgfkeysvalueof{/pgf/pattern keys/angle},
 \pgfkeysvalueof{/pgf/pattern keys/line width},
 },
 bounding box={
 (0,-0.5*\pgfkeysvalueof{/pgf/pattern keys/line width}) and
 (\pgfkeysvalueof{/pgf/pattern keys/size},
 0.5*\pgfkeysvalueof{/pgf/pattern keys/line width})},
 tile size={(\pgfkeysvalueof{/pgf/pattern keys/size},
 \pgfkeysvalueof{/pgf/pattern keys/size})},
 tile transformation={rotate=\pgfkeysvalueof{/pgf/pattern keys/angle}},
 defaults={
 size/.initial=5pt,
 angle/.initial=45,
 line width/.initial=.4pt,
 },
 code={
 \draw [line width=\pgfkeysvalueof{/pgf/pattern keys/line width}]
 (0,0) -- (\pgfkeysvalueof{/pgf/pattern keys/size},0);
 },
 }

\begin{document}
\allowdisplaybreaks

\renewcommand{\thefootnote}{}

\newcommand{\arXivNumber}{2212.06526}

\renewcommand{\PaperNumber}{020}

\FirstPageHeading

 \ShortArticleName{Planar Orthogonal Polynomials as Type I Multiple Orthogonal Polynomials}

 \ArticleName{Planar Orthogonal Polynomials\\ as Type I Multiple Orthogonal Polynomials\footnote{This paper is a~contribution to the Special Issue on Evolution Equations, Exactly Solvable Models and Random Matrices in honor of Alexander Its' 70th birthday. The~full collection is available at \href{https://www.emis.de/journals/SIGMA/Its.html}{https://www.emis.de/journals/SIGMA/Its.html}}}

\Author{Sergey BEREZIN~$^{\rm ab}$, Arno B.J.~KUIJLAARS~$^{\rm a}$ and Iv\'{a}n PARRA~$^{\rm a}$}

\AuthorNameForHeading{S.~Berezin, A.B.J.~Kuijlaars and I.~Parra}

\Address{$^{\rm a)}$~Department of Mathematics, Katholieke Universiteit Leuven,\\
\hphantom{$^{\rm a)}$}~Celestijnenlaan 200B box 2400, 3001 Leuven, Belgium}
\EmailD{\href{mailto:sergey.berezin@kuleuven.be}{sergey.berezin@kuleuven.be}, \href{mailto:arno.kuijlaars@kuleuven.be}{arno.kuijlaars@kuleuven.be}, \href{mailto:ivan.parra@kuleuven.be}{ivan.parra@kuleuven.be}}

\Address{$^{\rm b)}$~St.~Petersburg Department of V.A.~Steklov Mathematical Institute of RAS,\\
\hphantom{$^{\rm b)}$}~Fontanka~27, 191023 St.~Petersburg, Russia}
 \EmailD{\href{mailto:serberezin@math.huji.ac.il}{serberezin@math.huji.ac.il}}

\ArticleDates{Received December 14, 2022, in final form March 21, 2023; Published online April 12, 2023}

\Abstract{A recent result of S.-Y.~Lee and M.~Yang states that the planar orthogonal polynomials orthogonal with respect to a modified Gaussian measure are multiple orthogonal polynomials of type~II on a contour in the complex plane. We show that the same polynomials are also type~I orthogonal polynomials on a contour, provided the exponents in the weight are integer. From this orthogonality, we derive several equivalent Riemann--Hilbert problems. The proof is based on the fundamental identity of Lee and Yang, which we establish using a new technique.}


\Classification{42C05; 30E25; 41A21}

 \begin{flushright}
 \begin{minipage}{60mm}
 \textit{Dedicated to Alexander Its\\ on the occasion of his 70th birthday}
 \end{minipage}
 \end{flushright}

\renewcommand{\thefootnote}{\arabic{footnote}}
\setcounter{footnote}{0}

 \section{Introduction}
 This work is inspired by Lee and Yang's paper~\cite{LY2}, which showed that planar orthogonal polynomials can be viewed as multiple orthogonal polynomials of type~II on a contour in the complex plane. Their work extends an earlier result of Balogh, Bertola, Lee, and McLaughlin~\cite{BBLM}.

 We show that the same polynomials are also multiple orthogonal polynomials of type~I if the exponents in the weight are positive integers, unlike in the situation studied in~\cite{LY2}, where these exponents are arbitrary positive real numbers. We also present a novel, more transparent, technique to transform planar orthogonality into orthogonality on a contour. Before we begin, note that the title of our paper differs from the title of \cite{LY2} only in one letter, yet this makes a~considerable difference in the arguments used.

 The polynomials in question are orthogonal with respect to a modified Gaussian measure,
 \begin{gather}
 \label{eq:muW}
 \mu_W({\rm d}z) = \frac{1}{\pi} |W(z)|^2 {\rm e}^{-|z|^2} \Leb({\rm d}z),
 \end{gather}
 where $\Leb$ denotes the Lebesgue measure on~$\Cset$ (identified with~$\Rset^2$) and the weight~$W$ reads{\samepage
 \begin{equation}
 \label{eq:Wdef}
 W(z) = \prod_{j=1}^p (z-a_j)^{c_j}, \qquad z \in \Cset,
 \end{equation}
 where~$p \in \Nset$, the~$c_j$ are positive real numbers, and the~$a_j$ are distinct complex numbers.}

 If the~$c_j$ are not necessarily integer, one needs to specify the branch cuts and fix the branches in order to render~\eqref{eq:Wdef} unambiguous. This complicates the analysis, and we return to such a general scenario only episodically. In contrast, if all~$c_j$ are positive integers, $W$ becomes a~polynomial of degree~$c = \sum_{j=1}^p c_j$ and~\eqref{eq:Wdef} extends to the whole complex plane~$\mathbb{C}$. This is the situation of our primary concern.

 Denote the scalar product corresponding to~\eqref{eq:muW} by
 \begin{equation}
 \label{eq:fgscalar}
 \langle f, g \rangle_W = \int_{\Cset} f(z)\overline{g(z)} \, \mu_W({\rm d}z).
 \end{equation}
 Then, the $n$-th degree monic orthogonal polynomial $P_n$ with
 respect to~$\mu_W$ can be uniquely recovered by solving a linear system of equations for its coefficients,
 \begin{equation}
 \label{def:Pn}
 \int_{\Cset} P_{n}(z)\overline{z}^k\, \mu_W({\rm d}z)=0, \qquad k=0,\dots,n-1.
 \end{equation}

The motivation for studying planar orthogonal polynomials comes from the theory of non-Hermitian random matrices, in particular from those related to the normal matrix model. In~this model, the eigenvalues of an~$n\times n$ normal matrix have the joint density
 \begin{equation}
 \label{eigdensity}
 \frac{1}{Z_n} \prod_{j<k} |z_k-z_j|^2 \prod_{j=1}^n {\rm e}^{-V(z_j)},
 \end{equation}
 where~$V$ is the potential of the model and $Z_n$ is a normalization constant. The eigenvalues form a determinantal point process with the correlation kernel constructed in terms of the planar orthogonal polynomials orthogonal with respect to the one-particle weight~${\rm e}^{-V(z)}$. The case~\eqref{eq:muW}--\eqref{eq:Wdef} corresponds to
 \begin{equation}
 \label{Vp}
 V(z) = |z|^2 - 2 \sum_{j=1}^p c_j \log |z-a_j|.
 \end{equation}
 In particular, for integer~$c_j$'s the probability law corresponding
 to~\eqref{eigdensity}--\eqref{Vp} can be interpreted as that of a Ginibre ensemble
 of size $n+c$ conditioned on having an eigenvalue of
 multiplicity~$c_j$ at~$a_j$ for each~$j=1,\dots,p$.

 The determinantal structure in~\eqref{eigdensity} allows for a complete description of the eigenvalue correlation functions at the finite size~$n$ in terms of the correlation kernel, which in turn can be used to study the large~$n$ behavior of both the polynomials and the eigenvalues. In such studies, one typically replaces $V$ by $nV$ in~\eqref{eigdensity} to obtain a balance between the ``repulsion'' and ``confinement'' present in the determinantal model~\eqref{eigdensity}. We refer to the surveys \cite[Section 5]{BF}, \cite[Chapter 6]{GTV} and references therein for more information on the normal matrix model.


 In the analogous situation of Hermitian random matrices, the eigenvalue correlations are described by orthogonal polynomials on the real line. The theory of such polynomials is well-developed, and as a result the corresponding ensembles are understood much better than their non-Hermitian counterparts. One basic result is that the eigenvalues of Hermitian matrices and the zeros of the corresponding orthogonal polynomials (both real) have the same limiting behavior as~$n \to \infty$ (e.g., see~\cite{Deif}). More subtle results on the universality of local eigenvalue statistics were established using the characterization of orthogonal polynomials on the real line via a~$2\times2$ matrix-valued Riemann--Hilbert problem (e.g., see~\cite{FIK2}), followed by the Deift--Zhou steepest descent analysis (e.g., see~\cite{DKMVZ,DZ}).

 The planar case, on the other hand, is more intricate. One has to distinguish between the asymptotic behavior of the random eigenvalues governed by~\eqref{eigdensity} (with $V$ replaced by $nV$) and the limiting behavior of the zeros of the corresponding orthogonal polynomials.
 While it is known that the eigenvalues fill out a two-dimensional domain called the droplet, the understanding of the asymptotic behavior of the zeros of the planar orthogonal polynomials is rather limited.
 Results exist when the classical Hermite, Laguerre and Gegenbauer polynomials appear as the planar orthogonal polynomials
 (e.g., see \cite{ANPV,Karp,EM}), as well as for some special cases where the planar orthogonality can be reformulated as
 (multiple) orthogonality on a contour and the Riemann--Hilbert techniques can be used
 (e.g., see~\cite{BBLM,BGM,BEG,BK,DS,LY1}). For example, in~\cite{BBLM} the situation~\eqref{eq:muW}--\eqref{eq:Wdef}
 with $p=1$ is considered. In this case,
 the Hermitian planar orthogonality can be transformed to non-Hermitian orthogonality
 on a contour due to a special identity~\cite[Lemma 3.1]{BBLM}, and rigorous analysis
 is possible. The same identity was used in \cite{WW} in a study on moments of complex Ginibre matrices.

 Multiple orthogonality plays a role in \cite{BK,LY1,LY3}, where it can be treated by using large size Riemann--Hilbert problems (e.g., see~\cite{VAGK}).

 The common feature of the examples above is that the zeros accumulate along a one-dimensional curve
 (or a system of such curves) in the complex plane, known as the motherbody. We conjecture that this is
 a general phenomenon for all real analytic potentials, including those in~\eqref{Vp}. The
 results of~\cite{LY2, LY3} support this conjecture. Indeed, the planar orthogonality
 corresponding to~\eqref{eigdensity} with~$V$ given by
 \begin{equation*}
 V(z) = n |z|^2 - 2 \sum_{j=1}^p c_j \log |z-a_j|
 \end{equation*}
 is studied in~\cite{LY2,LY3} for the case of fixed~$c_j$'s independent of~$n$. The droplet turns out to be the unit disk, and the motherbody is supported on a multiple Szeg\H{o} curve that depends on the~$a_j$.

 In the scenario when the~$c_j$ grow linearly with~$n$ and $p \geq 2$, the Riemann--Hilbert problem in~\cite{LY3}
 has not been analyzed successfully yet. The multiple orthogonality of type~I that we discovered, as we will show,
 leads to several different Riemann--Hilbert problems. Our hope is that one of them will help to carry through with the
 steepest descent analysis.

 If all $c_j$ are integer valued, the planar orthogonal polynomials can also be
 expressed as ratios of determinants as shown in \cite{AV}. The determinants are growing in size as the
 $c_j$ increase, and therefore this determinantal formula may not be particularly useful for asymptotic analysis.

 We finally remark that the important work of Hedenmalm and Wennman~\cite{HW} provides the asymptotic behavior of the planar orthogonal polynomials in the exterior of the droplet and on its boundary (even slightly inside, under certain assumptions including those of the real analyticity of the boundary). This, however, does not give any information about the motherbody since it is inside the droplet (see~\cite[Remark 1.6\,(c)]{HW}).

 \section{Statement of result}

 Our main result is Theorem~\ref{thm:theo1} below. It gives a number of properties that are equivalent to the planar orthogonality corresponding to~\eqref{eq:muW}--\eqref{def:Pn} in the case the~$c_j$ are positive integers. Below, we use the conjugate~$W^*$ of~$W$ defined by
 \begin{equation}
 \label{eq:Wconj}
 W^*(z) = \overline{W(\bar{z})} = \prod_{j=1}^p (z-\overline{a_j})^{c_j}, \qquad z \in \Cset.
 \end{equation}
 We also use~$\mathscr{D}_z$ to denote the derivative operator with respect to~$z$, and then~$W^*(\mathscr{D}_z)$ is the differential operator
 \begin{equation}
 \label{eq:Woperator}
 W^*(\mathscr{D}_z) = \prod_{j=1}^p (\mathscr{D}_z - \overline{a_j})^{c_j}.
 \end{equation}

 \begin{Theorem}
 \label{thm:theo1}
 Let~$W$ be given by~\eqref{eq:Wdef} where all~$c_j$ are positive integers $\big($so that~$W$ is a~polynomial of degree~$c = \sum_{j=1}^p c_j\big)$, and let the~$a_j$ be distinct complex numbers. Then, the following properties are equivalent for a monic polynomial~$P_n$ of degree~$n$,
 \begin{enumerate}
 \item[$(a)$] $P_n$ is the planar orthogonal polynomial on~$\Cset$ with weight~\eqref{eq:muW}.
 \item[$(b)$] $P_n$ satisfies
 \begin{equation}
 \label{eq:Pnequiv2}
 \frac{1}{2\pi {\rm i}} \oint_{\gamma} P_n(z) W(z) \phi_k(z)\, {\rm d}z = 0, \qquad k=0,1, \dots, n-1,
 \end{equation}
 where $\gamma$ is a closed contour around the origin and
 \begin{equation}
 \label{eq:phikdef}
 \phi_k(z) = \int_0^{\bar{z} \times \infty} W^*(u) u^k {\rm e}^{-uz}\, {\rm d}u,
 \end{equation}
 where the path of integration in~\eqref{eq:phikdef} goes from~$0$ to~$\infty$ along the ray $\arg u = \arg \bar{z}$.
 \item[$(c)$] One has
 \begin{equation}
 \label{eq:Pnequiv3}
 W^*(\mathscr{D}_z) \left[ P_n(z) W(z) \right] = O(z^n)
 \end{equation}
 as~$z \to 0$.
 \item[$(d)$] There exist polynomials~$Q_j$ of~$\deg Q_j \leq c_j-1$, for $j=1, \dots, p$, such that
 \begin{equation}
 \label{eq:Pnequiv4}
 P_n(z) W(z) + \sum_{j=1}^p Q_j(z) {\rm e}^{\overline{a_j} z} = O\big(z^{n + c}\big)
 \end{equation}
 as~$z \to 0$.
 \end{enumerate}
 \end{Theorem}

 \begin{Remark}
 The property \eqref{eq:Pnequiv2} of the planar orthogonal polynomials has already been obtained by Lee and Yang~\cite{LY2}. They assume that the points~$a_j$ are distinct, non-zero, and with different arguments modulo~$2\pi$, and show that~\eqref{eq:Pnequiv2} leads to multiple orthogonality of type~II (see also Section~\ref{section5} below) for real positive~$c_j$'s. In this general situation, because of the branch cuts, the contour~$\gamma$ in~\eqref{eq:Pnequiv2} has to pass through all~$a_j$ and can no longer be an arbitrary contour around the origin.
 \end{Remark}

 The proof of Theorem~\ref{thm:theo1} relies on the \textit{fundamental identity} of Lee and Yang~\cite[Proposition~1]{LY2}. We state and prove it in the next section. Our proof is based on a new technique and is of independent interest. The proof in~\cite{LY2} (see also the proof of~\cite[Lemma 3.1]{BBLM}) goes as follows. One first restricts the integral in~\eqref{def:Pn} to a large disk~$D_R$ and then applies Stokes' theorem to rewrite the new integral over~$D_R$ as an integral over the boundary~$\partial D_R$. Then, \eqref{eq:Pnequiv2} follows by a contour deformation argument and by passing to the limit~$R \to \infty$.

 In our proof, we only rely on most basic and elementary facts of complex analysis and avoid the use of Stokes' theorem. We write the integral over~$\Cset$ in polar coordinates and, by analyticity, deform the angular integral to an integral over~$\gamma$. The final step is to switch the angular and the radial integrals by Fubini's theorem.

 Theorem~\ref{thm:theo1} is proved in Section~\ref{section3}. Part $(c)$ of Theorem~\ref{thm:theo1} is a very concise
 representation of the planar orthogonality in part $(a)$, and thus is of interest in its own right. The equivalence
 with part $(d)$ leads directly to the multiple orthogonality of type~I as we explain in Section~\ref{section4}.
 We show that this orthogonality leads to three different, though closely related, Riemann--Hilbert problems.
 The latter uniquely characterize the planar orthogonal polynomial~$P_n$ and the auxiliary polynomials~$Q_j$.
 We point out in passing that none of these Riemann--Hilbert problems turns out to be related to the Riemann--Hilbert
 problem in~\cite{LY2} (corresponding to type~II orthogonality) in a canonical way (see also Section~\ref{sec:conc}).
 In Section~\ref{section5}, we focus on the type~II multiple orthogonality of Lee and Yang~\cite{LY2}. We give the
 corresponding proof in the case of polynomial~$W$, which is essentially the same as the proof
 in~\cite{LY2} but more transparent since no branch cuts for~$W$ are necessary.

 \section{Fundamental identity}

 Akin to~\eqref{eq:Wconj}, we define the conjugate of a function~$Q$ by
 \begin{equation}
 \label{eq:Qconj}
 Q^*(z) = \overline{Q(\bar{z})}, \qquad z \in \Cset.
 \end{equation}
 If~$Q$ is analytic in a certain domain~$\Omega$, then~$Q^*$ is also analytic however in the conjugate domain~$\Omega^* = \{z\in \Cset \mid \overline{z} \in \Omega\}$. If~$Q$ is a polynomial, then $Q^*$ is the polynomial whose coefficients are the complex conjugates of those of the polynomial~$Q$.

 The following result is due to Lee and Yang~\cite[Proposition 1]{LY2}. As already stated above, we give a different proof. For the sake of clarity, we first deal with the case of polynomial~$W$.

 \begin{Proposition}
 \label{prop:Contour1}
 Let~$P$ and~$Q$ be polynomials and suppose that the~$c_j$ in~\eqref{eq:Wdef} are positive integers.
 Then, the fundamental identity holds,
 \begin{align}
 \langle P, Q \rangle_W &= \frac{1}{\pi} \int_{\Cset} P(z) \overline{Q(z)} |W(z)|^2 {\rm e}^{-|z|^2} \Leb({\rm d}z)\nonumber \\
 & = \frac{1}{2\pi {\rm i}} \oint_{\gamma} P(z) W(z) \int_{0}^{\bar{z} \times \infty} W^*(u) Q^*(u) {\rm e}^{-uz}\, {\rm d}u\, {\rm d}z, \label{eq:PQscalar}
 \end{align}
 where $\gamma$ is a simple closed contour that goes around the origin once in the counterclockwise direction and the path for the~$u$ integral goes from~$0$ to~$\infty$ along the ray~$\arg u = \arg \bar{z}$.
 \end{Proposition}
 \begin{proof}
 Write the left-hand side of~\eqref{eq:PQscalar} in polar coordinates,
 \begin{equation}
 \label{eq:PnContour2}
 \langle P, Q\rangle_W =\frac{1}{\pi {\rm i}} \int_0^\infty \oint_{C_r} P(z) Q^*(\overline{z}) |W(z)|^2 \frac{{\rm d}z}{z} r {\rm e}^{-r^2} {\rm d}r,
 \end{equation}
 where the~$z$-integral is taken along the circle~$C_r$ of radius~$r$ around the origin. Observe that~$\overline{z} = r^2/z$ for~$z \in C_r$. In view of~\eqref{eq:Qconj}, we can write the following chain of identities,
 \begin{equation*}
 |W(z)|^2 = W(z) W^*(\overline{z}) = W(z) W^*\big(r^2/z\big), \qquad z \in C_r.
 \end{equation*}
 Since also~$Q^*(\overline{z}) = Q^*(r^2/z)$ for $z \in C_r$, the formula~\eqref{eq:PnContour2} becomes
 \begin{equation}
 \label{eq:PnContour3}
 \langle P, Q\rangle_W =\frac{1}{\pi {\rm i}} \int_0^\infty \oint_{C_r} P(z) Q^*\big(r^2/z\big) W(z) W^*\big(r^2/z\big) \frac{{\rm d}z}{z} r {\rm e}^{-r^2} {\rm d}r.
 \end{equation}

 Recall that all~$c_j$ are positive integers, thus~$W$ and~$W^*$ are polynomials and the integrand in~\eqref{eq:PnContour3} is	meromorphic in~$z$ with a sole pole at~$z=0$. Then, by Cauchy's	theorem, the contour can be deformed from~$C_r$ to a contour~$\gamma$ that goes around the origin in the counterclockwise direction once and is independent of~$r$. We use Fubini's theorem to get
 \begin{equation}
 \label{eq:PnContour4}
 \langle P, Q\rangle_W =\frac{1}{\pi {\rm i}} \oint_{\gamma} P(z)W(z) \int_0^\infty Q^*\big(r^2/z\big) W^*(r^2/z) \frac{r {\rm e}^{-r^2}}{z}\,{\rm d}r\, {\rm d}z.
 \end{equation}
 Changing variables in the inner integral, $u=r^2/z$, we arrive at~\eqref{eq:PQscalar}.
 \end{proof}

 Our method of proof easily extends to a more general setting. Assume that the~$c_j$ are positive real numbers, not necessarily integer as before. As in~\cite{LY2}, we restrict ourselves to the case that all~$a_j$ are non-zero, distinct, and have different arguments modulo~$2\pi$. For convenience, order the~$a_j$ so that~$0 \leq \arg a_1 < \arg a_2 < \cdots < \arg a_p < 2 \pi$. We can still define~$W$ by the same formula~\eqref{eq:Wdef} as earlier, however it is imperative one restrict the domain by making cuts. Following \cite{LY2}, we choose to cut along the rays
 \begin{equation}
 \label{eq:def_B}
 B = \bigcup_{j=1}^p \{z \in \Cset \mid z = a_j t,\, t \ge 1 \}.
 \end{equation}

 The domain~$\Cset \setminus B$ is simply connected. Fixing a branch for the power functions in~\eqref{eq:Wdef} renders~$W$ analytic in this domain. Note that such a choice of the branches does not affect~$|W(z)|^2$, which is assumed to be extended by continuity to the whole complex plane~$\Cset$.

 To formulate the analogue of Proposition~\ref{prop:Contour1}, we will set
 \begin{equation}
 \label{eq:def_Omega}
 \Omega = \Cset \setminus \bigcup_{j=1}^p \{z \in \Cset \mid z = a_j t,\, t \ge 0 \},
 \end{equation}
 which is a union of sectors separated by the rays~$\arg{z} = \arg{a_j}$ for~$j=1, \dots, p$.

 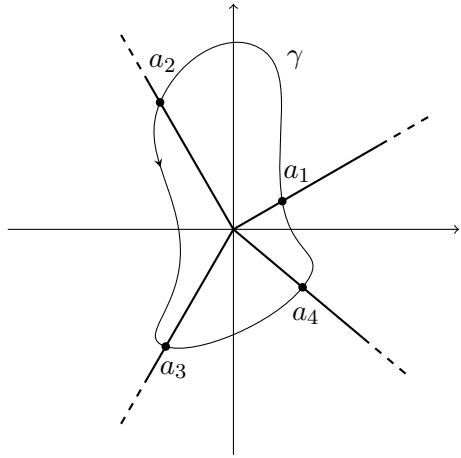
\begin{figure}[ht!]
 \centering
 \begin{tikzpicture}[scale=1.5]
 \begin{scope}
 \draw[->] (-2,0) -- (2,0);
 \draw[->] (0,-2) -- (0,2);
 \end{scope}
 \begin{scope}
 \draw[thick] (0,0) -- (30:1.5);
 \draw[thick,dashed] (30:1.5) -- (30:2);
 \draw[thick] (0,0) -- (120:1.5);
 \draw[thick,dashed] (120:1.5) -- (120:2);
 \draw[thick] (0,0) -- (240:1.5);
 \draw[thick,dashed] (240:1.5) -- (240:2);
 \draw[thick] (0,0) -- (320:1.5);
 \draw[thick,dashed] (320:1.5) -- (320:2);

 \node[draw,circle,inner sep=1pt,fill] at (30:0.5) {};
 \node[draw,circle,inner sep=1pt,fill] at (120:1.3) {};
 \node[draw,circle,inner sep=1pt,fill] at (240:1.2) {};
 \node[draw,circle,inner sep=1pt,fill] at (320:0.8) {};

 \node at (41:0.75) {$a_1$};
 \node at (113:1.6) {$a_2$};
 \node at (247:1.35) {$a_3$};
 \node at (310:1) {$a_4$};

 \node at (70:1.6) {$\gamma$};
 \end{scope}

 \begin{scope}[compass style/.style={color=black}, color=black, decoration={markings,mark= at position 0.25 with {\arrow{stealth}}}]
 \draw[postaction=decorate] plot [smooth cycle, tension=0.9] coordinates {(30:0.5) (80:1.6) (120:1.3) (200:0.5) (240:1.2) (320:0.8)};
 \end{scope}
 \end{tikzpicture}
 \caption{The domain~$\Omega$ in~\eqref{eq:def_Omega} and an example contour~$\gamma$.}
 \label{contour_gamma}
 \end{figure}

The following is~\cite[Proposition 1]{LY2}.

\begin{Proposition}\label{prop:Contour2}
 Suppose that the~$c_j$ are positive real numbers, but not necessarily integers, and the~$a_j$ are non-zero, distinct complex numbers that have different arguments modulo~$2 \pi$. Define~$\Omega$ as in~\eqref{eq:def_Omega}.\ Then~\eqref{eq:PQscalar} still holds, provided that~$\gamma$ is a counterclockwise-oriented~con\-tour in~$\Omega \cup \{a_1, \dots, a_p\}$, going around the origin once $($e.g., see Figure~$\ref{contour_gamma})$.
 \end{Proposition}

 \begin{Remark}\label{rem:pr2.2_1}
 Note that there is no need to make a cut for a certain~$a_j$ if the corresponding~${c_j \in \Nset}$. Nevertheless, we do so for the sake of notational convenience.
 \end{Remark}
 \begin{Remark}
 If~$a_1=0$, the proposition still holds with the same~$\Omega$ as in~\eqref{eq:def_Omega}; however, one needs to choose~$\gamma$ in~$\Omega \cup \{a_2, \dots, a_p\}$ instead. Indeed, first note that Remark~\ref{rem:pr2.2_1} is still applicable. Hence, without loss of generality~$c_1 \notin \Nset$. One needs to modify~$B$ in~\eqref{eq:def_B} by including an additional cut from~$0$ to~$\infty$, transversal to all the other cuts. Then, in a similar way as before, one can fix a branch of~$W(z)$ in the new simply-connected domain~$\Cset \setminus B$. Note that the definition of~$\Omega$ in~\eqref{eq:def_Omega} does not change, however one of the rays collapses to the point~$z=0$. Finally, observe that~$z^{c_1}$ effectively cancels from~\eqref{eq:gen_proof_fi_eq1}, and the proof goes through in the same way except the fact that~$Q^*(r^2/z)$ still may have a pole at zero and thus~$\gamma$ must not pass through~it. The latter explains why~$a_1=0$ was excluded from~$\Omega \cup \{a_1, \dots, a_p\}$.
 \end{Remark}

 \begin{proof}[Proof of Proposition~\ref{prop:Contour2}]
 If not all of the~$c_j$ are integers, then we need branch cuts to define~$W$ and~$W^*$. For every $r > 0$, we have
 \begin{equation}
 \label{eq:gen_proof_fi_eq1}
 W(z) W^*\big(r^2/z\big) = \prod_{j=1}^p (z-a_j)^{c_j} \Bigg(\frac{\frac{r^2}{|a_j|^2} a_j-z}{z/\overline{a_j}} \Bigg)^{c_j},
 \end{equation}
 which is initially only analytic in~$\Omega$. If~$r \neq |a_j|$ for every~$j=1, \dots, p$, then there is an analytic continuation across the open straight line segment from~$a_j$ to~$\frac{r^2}{|a_j|^2} a_j$ along the ray~$\arg z = \arg a_j$ for every~$j=1, \dots, p$. That is, $W(z) W^*\big(r^2/z\big)$ has an analytic continuation to
 \begin{equation*}
 \Omega_r := \Omega \cup \bigg\{ (1-t) a_j + t \frac{r^2}{|a_j|^2} a_j \,\bigg|\, 0 < t < 1\bigg\}.
 \end{equation*}
 Observe that~$\tfrac{r^2}{|a_j|^2} a_j$ is the image of~$a_j$ under the reflection about the circle~$C_r$. Thus, $C_r$ is contained in~$\Omega_r$.

 With these preparations in mind, we follow the proof of Proposition~\ref{prop:Contour1}. The identity~\eqref{eq:PnContour3} still holds. Then, by the above and Cauchy's theorem, for each~$r > 0$ we are allowed to deform~$C_r$ to a contour~$\gamma$ as in the statement. Since~$\gamma$ is independent of $r$, we can apply Fubini's theorem to~\eqref{eq:PnContour3}. This yields~\eqref{eq:PnContour4}, and we can proceed as in the proof of Proposition~\ref{prop:Contour1}.
 \end{proof}

 Note that it is not difficult to extend the proofs of Propositions~\ref{prop:Contour1} and~\ref{prop:Contour2} to other types of weights, mutatis-mutandis. For instance, if
 \begin{equation*}
 \mu({\rm d}z) = \frac{1}{\pi} |W(z)|^2 \frac{\Leb({\rm d}z)}{\big(1+|z|^2\big)^{\alpha}}, \qquad \alpha> 1+ c,
 \end{equation*}
 (which generalizes spherical ensembles, e.g., see~\cite[Section~2.5]{BF}), one gets
 \begin{equation*}
 \langle P, Q \rangle_W = \frac{1}{2\pi {\rm i}} \oint_{\gamma} P(z) W(z) \int_{0}^{\bar z \times \infty} \frac{Q^*(u) W^*(u)}{(1+uz)^{\alpha}}\,{\rm d}u \, {\rm d}z,
 \end{equation*}
 which holds as long as~$\deg P + \deg Q < 2 \alpha - 2c-2$ so that the integral over~$\Cset$ converges.

 \section{Proof of Theorem~\ref{thm:theo1}}
 \label{section3}
 \begin{proof}
 $(a) \Leftrightarrow (b)$: Taking~$Q(z) = z^k$ in the identity~\eqref{eq:PQscalar} and recalling the definition~\eqref{eq:phikdef} of~$\phi_k$, we obtain
 \begin{equation}
 \label{eq:Pzk}
 \big\langle P, z^k \big\rangle_W = \frac{1}{2\pi {\rm i}} \oint_\gamma P(z) W(z) \phi_k(z)\, {\rm d}z.
 \end{equation}
 This identity \eqref{eq:Pzk} shows that $(a)$ and $(b)$ in Theorem~\ref{thm:theo1}
 are equivalent.

\smallskip\noindent $(b) \Leftrightarrow (c)$:
 From basic properties of the Laplace transform, we get
 \begin{equation*}
 \int_{0}^{\bar z \times \infty} W^*(u) u^k {\rm e}^{-uz}\, {\rm d}u = W^*(-\mathscr{D}_z) (-\mathscr{D}_z)^k \bigg[\frac{1}{z}\bigg] = W^*(-\mathscr{D}_z) \bigg[\frac{k!}{z^{k+1}}\bigg],
 \end{equation*}
 Thus, recalling~\eqref{eq:phikdef}, we can write
 \begin{equation*}
 \frac{1}{2\pi {\rm i}} \oint_{\gamma} P_n(z) W(z) \phi_k(z)\, {\rm d}z = \frac{k!}{2\pi {\rm i}} \oint_{\gamma} P_n(z)W(z) W^*(-\mathscr{D}_z) \bigg[\frac{1}{z^{k+1}}\bigg] {\rm d}z
 \end{equation*}
 for any polynomial $P_n$. Integrating by parts, we find
 \begin{equation*}
 \frac{1}{2\pi {\rm i}} \oint_{\gamma} P_n(z) W(z) \phi_k(z) {\rm d}z = \frac{k!}{2\pi {\rm i}} \oint_{\gamma} \frac{W^*(\mathscr{D}_z)[P_n(z)W(z)]}{z^{k+1}} \, {\rm d}z,
 \end{equation*}
 as there are no boundary terms on the closed contour~$\gamma$ and~$W^*(\mathscr{D}_z)$ is the adjoint of~$W^*(-\mathscr{D}_z)$. Thus, $P_n$ satisfies~\eqref{eq:Pnequiv2} if and only if
 \begin{equation}
 \label{eq:Taylorzero2}
 \frac{1}{2\pi {\rm i}}
 \oint_{\gamma} \frac{W^*(\mathscr{D}_z)[P_n(z)W(z)]}{z^{k+1}} \, {\rm d}z =0, \qquad k =0, \dots, n-1.
 \end{equation}
 The identities~\eqref{eq:Taylorzero2} mean that the coefficient before~$z^k$ of the polynomial~$W^*(\mathscr{D}_z)[P_n(z)W(z)]$ vanishes for~$k=0, \dots, n-1$, and we conclude that items $(b)$ and $(c)$ are equivalent.

 In the proof that $(c)$ and $(d)$ are equivalent, we are going to use three basic properties of the operator~$W^*(\mathscr{D}_z)$ from~\eqref{eq:Woperator}. These properties are stated separately for the ease of reference.

 \medskip\noindent
 \textit{Property $W_1$.} $W^*(\mathscr{D}_z)$ is a linear differential operator with kernel
 \begin{equation}
 \label{eq:Wproperty1}
 \ker W^*(\mathscr{D}_z) = \Bigg\{ \sum_{j=1}^p Q_j(z) {\rm e}^{\overline{a_j} z} \, \bigg|\, \deg{Q_j} \leq c_j-1 \text{ for } j=1, \dots, p \Bigg\}.
 \end{equation}
 The kernel is a vector space over~$\Cset$ and its dimension is~$c = \sum_{j=1}^p c_j$.

 \medskip\noindent
 \textit{Property $W_2$.}
 If~$P$ is a polynomial then~$W^*(\mathscr{D}_z) [P(z)]$ is a polynomial and
 \begin{equation}
 \label{eq:Wproperty2}
 \deg W^*(\mathscr{D}_z) [P(z)] = \deg P(z),
 \end{equation}
 provided all~$a_j \neq 0$. If~$a_1 =0$ and~$\deg P(z) \geq c_1$, then
 \begin{equation}
 \label{eq:Wproperty2alt}
 \deg W^*(\mathscr{D}_z) [P(z)] = \deg P(z)-c_1.
 \end{equation}

 \medskip\noindent
 \textit{Property $W_3$.} $W^*(\mathscr{D}_z)$ reduces the order of vanishing at~$z=0$ by~$c = \deg W$, provided the order of vanishing is greater than or equal to~$c$. Namely, for every non-negative integer~$n$, and analytic function $F$ with $F(z) = O(z^{n+c})$ as $z \to 0$, we have
 \begin{equation}
 \label{eq:Wproperty3}
 W^*(\mathscr{D}_z) [ F(z) ] = O(z^n)
 \end{equation}
 as~$z \to 0$.

 \smallskip\noindent
 $(d) \Rightarrow (c)$:
 Suppose~$P_n$ satisfies~\eqref{eq:Pnequiv4}. Then, by applying~$W^*(\mathscr{D}_z)$ and using Properties~$W_1$ and~$W_3$, we obtain
 \begin{equation}
 W^*(\mathscr{D}_z) [P_n(z) W(z)] = W^*(\mathscr{D}_z) \big[ O\big(z^{n+c}\big)\big] = O(z^n),
 \end{equation}
 which is~\eqref{eq:Pnequiv3}. Thus, $(d)$ implies $(c)$.

 \smallskip\noindent
$(c) \Rightarrow (d)$:
 In the proof, we assume without loss of generality that~$a_1 =0$ and~$c_1 \geq 0$. This includes the case that~$W$ has no zero at the origin since one can always set~$c_1=0$.

 Consider the linear mapping~$\pi\colon \ker(W^*(\mathscr{D}_z)) \to \Cset^c$ given by
 \begin{equation}
 \label{eq:pidefalt}
 \ker(W^*(\mathscr{D}_z)) \ni Q(z) = \sum_{j=0}^{\infty} q_j z^j \stackrel{\pi}{\longmapsto} (q_0, \dots, q_{c_1-1}, q_{n+c_1}, \dots, q_{n+c-1}) \in \Cset^c.
 \end{equation}
 We claim that~$\pi$ is injective, which will imply that~$\pi$ is surjective since it is a linear mapping between vector spaces of the same dimension~$c$, as follows from~\eqref{eq:pidefalt} and Property~$W_1$. Suppose~$Q \in \ker(W^*(\mathscr{D}_z))$ with~$\pi(Q) = (0, \dots, 0) \in \mathbb C^c$. Then,
 \begin{equation}
 \label{eq:QzOncalt}	
 Q(z) = \sum_{j=c_1}^{n+c_1-1} q_j z^j + O\big(z^{n+c}\big)
 \end{equation}
 as~$z \to 0$.
 From~\eqref{eq:Wproperty3} and~$W^*(\mathscr{D}_z)[Q(z)] = 0$, it then follows that
 \begin{equation}
 \label{eq:WstarOznalt}
 W^*(\mathscr{D}_z)\Bigg[ \sum_{j=c_1}^{n+c_1-1} q_j z^j\Bigg] = O(z^n)
 \end{equation}
 as $z \to 0$. Because of~\eqref{eq:Wproperty2}--\eqref{eq:Wproperty2alt}, we have
 that the left-hand side of~\eqref{eq:WstarOznalt} is a polynomial of degree~$\leq n-1$,
 and due to~\eqref{eq:WstarOznalt} it has a zero at~$z=0$ of order at least~$n$. Hence, \eqref{eq:WstarOznalt} vanishes identically, and thus~$\sum_{j=c_1}^{n+c_1-1} q_j z^j$
 belongs to the kernel of~$W^*(\mathscr{D}_z)$. By Property~$W_1$, the kernel contains polynomials up to degree $c_1-1$ but no polynomials of higher degrees. Hence, $q_j=0$ for $j=c_1, \dots, n+c_1-1$. Consequently, by~\eqref{eq:QzOncalt}
 we have that~$Q(z) = O(z^{n+c})$ as~$z \to 0$.
 In particular, $Q(0) = Q'(0) = \cdots = Q^{(c-1)}(0) = 0$.
 Thus~$Q \in \ker(W^*(\mathscr{D}_z))$ is a solution of a~homogeneous constant coefficient linear ODE of order $c$ with $c$ vanishing initial conditions.
The uniqueness theorem for such ODEs then yields~$Q \equiv 0$, which justifies the claim.

 Now, assume that~$P_n$ satisfies~\eqref{eq:Pnequiv3}, and write
 \begin{equation*}
 P_n(z) W(z) = \sum_{j=0}^{n+c} p_j z^j.
 \end{equation*}
 Since~$\pi$ is surjective, there is~$Q \in \ker(W^*(\mathscr{D}_z))$,
 \begin{equation*}
 Q(z) = \sum_{j=0}^{\infty} q_j z^j,
 \end{equation*}
 such that
$
 \pi(Q) = (-p_0, \dots, -p_{c_1-1}, -p_{n+c_1}, \dots, -p_{n+c-1}).
$
 Then,
 \begin{equation}
 \label{eq:PnWQalt}
 P_n(z) W(z) + Q(z) = \sum_{j=c_1}^{n+c_1-1} (p_j + q_j) z^j + O\big(z^{n+c}\big).
 \end{equation}
 Applying~$W^*(\mathscr{D}_z)$ to~\eqref{eq:PnWQalt} and using~\eqref{eq:Pnequiv3}, $W^*(\mathscr{D}_z)[Q(z)]=0$, and~\eqref{eq:Wproperty3}, we obtain
 \begin{equation}
 \label{eq:Wpjqjalt}
 W^*(\mathscr{D}_z)\Bigg[\sum_{j=c_1}^{n+c_1-1} (p_j + q_j) z^j\bigg] = O(z^{n})
 \end{equation}
 as $z \to 0$. We find ourselves in the situation similar to that while proving the injectivity of~$\pi$. It followed from~\eqref{eq:WstarOznalt} that~$q_j = 0$ for~$j=c_1,\dots, n+c_1-1$. In the exactly same way, now it follows from~\eqref{eq:Wpjqjalt} that~$p_j+q_j = 0$ for~$j=c_1, \dots, n+c_1-1$.

 Due to~\eqref{eq:PnWQalt}, we obtain
 \begin{equation*}
 P_n(z) W(z) + Q(z) = O\big(z^{n+c}\big),
 \end{equation*}
 which is exactly~\eqref{eq:Pnequiv4} because~$Q \in \ker(W^*(\mathscr{D}_z))$ must be of the form given in~\eqref{eq:Wproperty1}. This shows that $(c)$ implies $(d)$ and completes the proof of Theorem~\ref{thm:theo1}.
 \end{proof}

 We present an alternative proof of~$(a) \Rightarrow (c)$ by making use of the Fourier transform of the measure $\mu_W$ in~\eqref{eq:muW},
 \begin{equation}
 \label{eq:muhat}
 \widehat{\mu}_W(\zeta, \overline{\zeta}) = \int_{\Cset} {\rm e}^{\overline{\zeta} z + \zeta \bar{z}} \mu_W({\rm d}z).
 \end{equation}

 \begin{proof}[Alternative proof]
 $(a) \Rightarrow (c)$:
 Apply the Wirtinger derivatives
 \begin{equation*}
 \mathscr{D}_{\zeta} = \frac{1}{2}\bigg(\frac{\partial}{\partial u} - {\rm i} \frac{\partial}{\partial v}\bigg), \qquad
 \mathscr{D}_{\overline{\zeta}} = \frac{1}{2}\bigg(\frac{\partial}{\partial u} + {\rm i} \frac{\partial}{\partial v}\bigg),
 \end{equation*}
 where~$\zeta = u + {\rm i}v$, to~\eqref{eq:muhat}, and observe that
 \begin{equation}
 \label{eq:Dmuhat}
 \mathscr{D}_{\overline{\zeta}}^j \mathscr{D}_{\zeta}^k \big[\widehat{\mu}_W(\zeta, \overline{\zeta})\big] = \int_{\Cset} z^j \overline{z}^k {\rm e}^{\overline{\zeta} z + \zeta \overline{z}} \mu_W({\rm d}z).
 \end{equation}

 Due to linearity, \eqref{eq:Dmuhat} leads to
 \begin{equation*}
 P_n(\mathscr{D}_{\overline{\zeta}}) \mathscr{D}_{\zeta}^k \big[ \widehat{\mu}_W(\zeta, \overline{\zeta}) \big] = \int_{\Cset} P_n(z) \overline{z}^k {\rm e}^{\overline{\zeta} z + \zeta \overline{z}} \mu_W({\rm d}z).
 \end{equation*}
 Assuming that~$P_n$ is the planar orthogonal polynomial satisfiying~\eqref{def:Pn}, we have
 \begin{equation}
 \label{eq:remeq2}
 P_n(\mathscr{D}_{\overline{\zeta}}) \mathscr{D}_{\zeta}^k \big[ \widehat{\mu}_W(\zeta,\overline{\zeta}) \big] \big|_{\zeta=0, \overline{\zeta} = 0}= 0, \qquad k=0,\dots,n-1,
 \end{equation}
 where we view $\zeta$ and $\overline{\zeta}$ as two independent variables. The latter is permissible as long as the Fourier transform~$\widehat{\mu}_W(\cdot,\cdot)$, as a function of two complex variables, extends analytically to the neighborhood of~$\big\{(\xi_1,\xi_2) \in \Cset \times \Cset\mid \xi_2=\overline{\xi_1}\big\}$. This is the case we encounter below.

 For~$\mu_W$ from~\eqref{eq:muW} with polynomial~$W$, the Fourier transform can be computed as follows,
 \begin{align}
 \label{eq:hatmuW}
 \widehat{\mu}_W(\zeta,\bar{\zeta}) &=\frac{1}{\pi} \int_{\Cset} {\rm e}^{ \overline{\zeta} z + \zeta \overline{z}} |W(z)|^2 {\rm e}^{-|z|^2} \Leb{({\rm d}z)}\nonumber
 \\
 &= \frac{1}{\pi} W(\mathscr{D}_{\overline{\zeta}})W^*(\mathscr{D}_\zeta) \int_{\Rset^2} {\rm e}^{2(u x + v y)} {\rm e}^{-x^2 -y^2}\, {\rm d}x\,{\rm d}y \qquad
\nonumber
 \\
 &= W(\mathscr{D}_{\overline{\zeta}})W^*(\mathscr{D}_\zeta) \big[{\rm e}^{\zeta \overline{\zeta}} \big],
 \end{align}
 where $\zeta = u + {\rm i}v$, $z=x+ {\rm i}y$.

 Inserting~\eqref{eq:hatmuW} into~\eqref{eq:remeq2}, and changing the order of the differential operators, we obtain
 \begin{equation}
 \label{eq:remeq3}
 \mathscr{D}_{\zeta}^k W^*(\mathscr{D}_\zeta) (P_n W)(\mathscr{D}_{\overline{\zeta}}) \big[{\rm e}^{\zeta \overline{\zeta}}\big] \big|_{\zeta=0, \bar{\zeta} =0}= 0, \qquad k=0,\dots,n-1.
 \end{equation}
 Since
 \begin{equation*}
 (P_n W)(\mathscr{D}_{\overline{\zeta}}) \big[{\rm e}^{\zeta \overline{\zeta}} \big]	\big|_{\overline{\zeta}=0} = P_n(\zeta) W(\zeta),
 \end{equation*}
 we see that~\eqref{eq:remeq3} implies
 \begin{equation*}
 \mathscr{D}_{\zeta}^k W^*(\mathscr{D}_\zeta) [P_n(\zeta)W(\zeta)] \big|_{\zeta=0} = 0,\qquad k=0,\dots,n-1,
 \end{equation*}
 which is equivalent to~\eqref{eq:Pnequiv3}.
 \end{proof}

 This alternative proof will work for other types of measures~$\mu$ as long as the Fourier transform~$\widehat{\mu}$ has a simple enough representation. On the other hand, non-polynomial weights~$W$ will require the use of fractional derivatives, which are known to be non-local operators, causing substantial complications in this proof.

 \section{Type I multiple orthogonality}
 \label{section4}

 The property~\eqref{eq:Pnequiv4} in part $(d)$ of Theorem~\ref{thm:theo1} can be regarded as a Hermite--Pad\'{e} approximation problem of type~I at the origin. The general form of such approximation problems is the following. Given a collection of analytic functions (or, more generally, formal power series)~$f_0, \dots, f_p$ at~$z=0$ and a multi-index~$\vec{n} = (n_0, \dots, n_p) \in \Nset^{p+1}$, find polynomials~$Q_0, Q_1, \dots, Q_p$ of degrees~$\deg Q_j \leq n_j-1$, $j=0,1, \dots, p$, such that
 \begin{equation*}
 \sum_{j=0}^p Q_j(z) f_j(z) = O\big(z^{|\vec{n}|-1}\big)
 \end{equation*}
 as~$z \to 0$, where~$|\vec{n}| = \sum_{j=0}^p n_j$. For more information on Hermite--Pad\'{e} approximation and multiple orthogonal polynomials, we refer to~\cite{VA} and references therein.

 The problem~\eqref{eq:Pnequiv4} becomes a Hermite--Pad\'{e} type~I approximation problem
 for polynomials~$Q_0=P_n, Q_1, \dots, Q_p$, corresponding
 to the weights~$f_0(s)=W(s),f_1(s)={\rm e}^{\overline{a_1}s}, \dots, f_p(s)={\rm e}^{\overline{a_p}s}$
 and to the multi-index $\vec{n} = (n+1,c_1,\dots,c_p)$.
 This leads directly to the type~I multiple orthogonality. Indeed, the function on the
 left-hand side of~\eqref{eq:Pnequiv4} has vanishing Taylor coefficients up to and
 including order~$n+c-1$. By the Cauchy integral formula, this tells us that~\eqref{eq:Pnequiv4} is equivalent to
 \begin{equation}
 \label{eq:Pnequiv5}
 \frac{1}{2\pi {\rm i}} \oint_\gamma \Bigg(P_n(s) \frac{W(s)}{s^{n+c}} + \sum_{j=1}^p Q_j(s) \frac{{\rm e}^{\overline{a_j} s}}{s^{n+c}}\Bigg) s^k\, {\rm d}s = 0,
 \qquad k=0, \dots, n+c-1,
 \end{equation}
 where~$\gamma$ is a simple closed contour around the origin. This is multiple orthogonality of type~I on the contour~$\gamma$ with~$p+1$ weight functions~$W(s)/s^{n+c}, {\rm e}^{\overline{a_1}s}/s^{n+c}, \dots, {\rm e}^{\overline{a_p}s}/s^{n+c}$ which are meromorphic with a sole pole at~$s=0$.

 Both multiple orthogonality of type~I and type~II are characterized by Riemann--Hilbert problems that were identified by Van Assche, Geronimo, and Kuijlaars in~\cite{VAGK}. For the particular case~\eqref{eq:Pnequiv5}, the Riemann--Hilbert problem is of size~$(p+2) \times (p+2)$ and its jump is on the contour~$\gamma$. Below, we will use the following notation. For any oriented contour~$\gamma$ and a function~$Y$ defined in~$\Cset \setminus \gamma$, we write~$Y_+$ and~$Y_-$ to denote the limiting values of~$Y$ on~$\gamma$ from the left and right sides, respectively.

 \begin{rhproblem}[type I multiple orthogonality for~$P_n,Q_1,\dots,Q_p$]
 Find solution $Y\colon \Cset \setminus \gamma \to \Cset^{(p+2) \times (p+2)}$
 such that
 \begin{enumerate}
 \item[$(1)$] $Y(z)$ is analytic for~$z \in \Cset \setminus \gamma$;

 \item[$(2)$] $Y_+(z) = Y_-(z) J_Y(z)$ for $z \in \gamma$, where
$
 J_Y(z) =
 \begin{bmatrix}
 1 & 0&\cdots &0& W(z)/z^{n+c}\\
 0 & 1& \cdots &0& {\rm e}^{\overline{a_1} z}/z^{n+c}\\
 \vdots & \vdots&\ddots&\vdots&\vdots\\
 0 & 0& \cdots &1& {\rm e}^{\overline{a_p} z}/z^{n+c}\\
 0&0&\cdots&0& 1
 \end{bmatrix}$;

 \item[$(3)$] $Y(z) = \big(I + O\big(\frac{1}{z}\big)\big)
 \begin{bmatrix}
 z^n&0&\cdots&0&0\\
 0&z^{c_1}&\cdots& 0&0\\
 \vdots&\vdots& \ddots&\vdots&\vdots\\
 0&0&\cdots&z^{c_p}&0\\
 0&0&\cdots&0&z^{-(n+c)}
 \end{bmatrix}
 $ as~$z \to \infty$.
 \end{enumerate}
 \end{rhproblem}

 The unique solution of the Riemann--Hilbert problem has the polynomials~$P_n$, $Q_1, \dots, Q_p$ in its first row. Namely,
 \begin{equation*}
 Y(z)=
 \begin{bmatrix}
 P_n(z) & Q_1(z)&\cdots& Q_p(z)& \displaystyle \frac{1}{2 \pi {\rm i}} \oint_\gamma \Bigg(P_n(s) \frac{W(s)}{s^{n+c}} + \sum_{j=1}^pQ_j(s)\frac{{\rm e}^{\overline{a_j}s}}{s^{n+c}}\Bigg)\frac{{\rm d}s}{s-z}\\
 * & * & \cdots &*&*\\
 \vdots & \vdots & \ddots & \vdots&\vdots\\
 * & * & \dots &*&*
 \end{bmatrix}
 \end{equation*}
 for~$z \in \Cset \setminus \gamma$. The other rows are filled with type I multiple orthogonal polynomials of slightly different degrees (for details, see~\cite{VAGK}).

 One can reduce the size of the Riemann--Hilbert problem as follows. Divide both sides of~\eqref{eq:Pnequiv4} by~$W$ and carry over~$P_n$ to the right-hand side. Then, we obtain
 \begin{equation}
 \label{eq:Pnequiv6}
 \sum_{j=1}^p Q_j(z) \frac{{\rm e}^{\overline{a_j}z}}{W(z)} = -P_n(z) + O\big(z^{n+c}\big),
 \end{equation}
 where we assume that~$a_j \neq 0$ for all~$j$. Then by Cauchy's formula we obtain
 \begin{equation}
 \label{eq:Pnequiv7}
 \frac{1}{2\pi {\rm i}}\oint_\gamma \sum_{j=1}^p Q_j(s) \frac{{\rm e}^{\overline{a_j} s}}{s^{n+c} W(s)}\ s^k\, {\rm d}s = - \delta_{k,c-1}, \qquad k=0, \dots, c-1,
 \end{equation}
 where~$\gamma$ is a counterclockwise-oriented closed contour going around the origin once, such that all the zeros of~$W$ lie in its exterior. This formulation of the type~I multiple orthogonality leads to a Riemann--Hilbert problem of size~$(p+1) \times (p+1)$.

 Note that we use~$Y$ again to denote the solution of the Riemann--Hilbert problem, although this solution is different from the earlier one that we also called~$Y$. We trust that this does not lead to any confusion.

 \begin{rhproblem}[type I multiple orthogonality for~$Q_1, \dots, Q_p$]
 Find solution $Y\colon \Cset \setminus \gamma \to \Cset^{(p+1) \times (p+1)}$ such that
 \begin{enumerate}
 \item[$(1)$] $Y(z)$ is analytic for~$z \in \Cset \setminus \gamma$;
 \item[$(2)$] $Y_+(z) = Y_-(z) J_Y(z)$ for $z \in \gamma$, where
$
 J_Y(z) =
 \begin{bmatrix}
 1 & \cdots &0& {\rm e}^{\overline{a_1} z}/(z^{n+c} W(z))\\
 \vdots &\ddots&\vdots&\vdots\\
 0 & \cdots &1& {\rm e}^{\overline{a_p} z}/(z^{n+c} W(z))\\
 0&\cdots&0& 1
 \end{bmatrix}$;

 \item[$(3)$] $Y(z) = \big(I + O\big(\frac{1}{z}\big)\big)
 \begin{bmatrix}
 z^{c_1}&\cdots&\cdots& 0\\
 \vdots& \ddots&\vdots&\vdots\\
 0&\cdots&z^{c_p}&0\\
 0&\cdots&0&z^{-c}
 \end{bmatrix}$ as $z \to \infty$.
 \end{enumerate}
 \end{rhproblem}

 The unique solution to this Riemann--Hilbert problem has the polynomials~$Q_1, \dots, Q_p$ in its last row,
 \begin{equation*}
 Y(z)=
 \begin{bmatrix}	
 * & \cdots &*&*\\
 \vdots & \ddots & \vdots & \vdots\\
 * & \cdots &*&* \\
 Q_1(z)&\cdots& Q_p(z)& \displaystyle \frac{1}{2 \pi {\rm i}} \oint_\gamma \sum _{j=1}^pQ_j(s)\frac{{\rm e}^{\overline{a_j}s}}{s^{n+c} W(s)}\frac{{\rm d}s}{s-z}
 \end{bmatrix}\!,
 \end{equation*}
 for~$z \in \Cset \setminus \gamma$. Once we know~$Q_1, \dots, Q_p$, then~$P_n$ can be recovered from~\eqref{eq:Pnequiv6}. That is, $-P_n$ is the~$n$-th partial sum of the Maclaurin series of~$\sum_{j=1}^p Q_j(z) \frac{{\rm e}^{\overline{a_j} z}}{W(z)}$.

 As a remark we add that if~$a_1=0$, then the O-term in~\eqref{eq:Pnequiv6} changes to~$O\big(z^{n+c-c_1}\big)$. The condition~\eqref{eq:Pnequiv7} and the Riemann--Hilbert problem similar to the above can still be written after appropriate modifications.

 In a similar manner, we can single out one of the~$Q_j$ in~\eqref{eq:Pnequiv4}. For the ease of notation, let us choose~$Q_p$ and rewrite~\eqref{eq:Pnequiv4} as
 \begin{equation*}
 P_n(z) W(z) {\rm e}^{-\overline{a_p}z} + \sum_{j=1}^{p-1} Q_j(z) {\rm e}^{(\overline{a_j} - \overline{a_p})z} = - Q_p(z) + O\big(z^{n+c}\big).
 \end{equation*}
 Then, by Cauchy's integral formula,
 \begin{align}
 \label{eq:Pnequiv8}
 \frac{1}{2\pi {\rm i}} \oint _{\gamma} \Bigg(P_n(s) \frac{W(s) {\rm e}^{-\overline{a_p} s}}{s^{n+c}} + \sum_{j=1}^{p-1} Q_j(s) \frac{{\rm e}^{(\overline{a}_j-\overline{a_p}) s}}{s^{n+c}}\Bigg) s^k\, {\rm d}s=0
 \end{align}
 with $k=0, \dots, n+c-c_{p}-1$, where $\gamma$ is an arbitrary closed contour around the origin.

 The corresponding Riemann--Hilbert problem is as follows.
 \begin{rhproblem}[type I multiple orthogonality for~$P_n,Q_1,\dots,Q_{p-1}$]
 Find solution $Y\colon \Cset \setminus \gamma \to \Cset^{(p+1) \times (p+1)}$
 such that
 \begin{enumerate}
 \item[$(1)$] $Y(z)$ is analytic for~$z \in \Cset \setminus \gamma$;

 \item[$(2)$] $Y_+(z) = Y_-(z) J_Y(z)$, $z \in \gamma$,
$
 \begin{bmatrix}
 1 & 0&\cdots &0& W(z){\rm e}^{-\overline{a_p} z}/z^{n+c}\\
 0 & 1& \cdots &0& {\rm e}^{(\overline{a_1}-\overline{a_p}) z}/z^{n+c}\\
 \vdots & \vdots&\ddots&\vdots&\vdots\\
 0 & 0& \cdots &1& {\rm e}^{(\overline{a_{p-1}} -\overline{a_p})z}/z^{n+c}\\
 0&0&\cdots&0& 1
 \end{bmatrix}
$;

 \item[$(3)$] $Y(z) = \big(I + O\big(\frac{1}{z}\big)\big)
 \begin{bmatrix}
 z^n&0&\cdots&0&0\\
 0&z^{c_1}&\cdots& 0&0\\
 \vdots&\vdots& \ddots&\vdots&\vdots\\
 0&0&\cdots&z^{c_{p-1}}&0\\
 0&0&\cdots&0&z^{-n-c+c_{p}}
 \end{bmatrix}
 $ as $z \to \infty$.
 \end{enumerate}
 \end{rhproblem}	

 The unique solution is
 \begin{gather*}
\setlength{\arraycolsep}{2.5pt}
 Y(z)= \begin{bmatrix}
 P_n(z) & Q_1(z)&\cdots&Q_{p-1}(z)& \displaystyle \frac{1}{2 \pi {\rm i}} \oint_\gamma \! \Bigg(\!P_n(s) \frac{W(s){\rm e}^{-\overline{a_p}z}}{s^{n+c}} +
 \sum_{j=1}^{p-1} Q_j(s) \frac{{\rm e}^{(\overline{a_j} - \overline{a_p}) s}}{s^{n+c}}\Bigg) \frac{{\rm d}s}{s\!-\!z}\\
 * & * & \cdots &*&*\\
 \vdots & \vdots & \ddots & \vdots&\vdots\\
 * & * & \cdots &*&*
 \end{bmatrix}
 \end{gather*}
 for~$z \in \Cset \setminus \gamma$.

 For $p=1$, we recover the Riemann--Hilbert problem for the orthogonal polynomials from~\cite{BBLM}. Indeed, if~$p=1$, $W(z)=(z-a)^c$, then~\eqref{eq:Pnequiv8} yields
 \begin{equation*}
 \int _{\gamma} \frac{P_n(s) W(s) {\rm e}^{-\overline{a} s}}{s^{n+c}} s^k\, {\rm d}s=0, \qquad k=0, \dots, n-1,
 \end{equation*}
 which is a usual non-Hermitian orthogonality on a contour. The above Riemann--Hilbert problem reduces to the usual Riemann--Hilbert problem for orthogonal polynomials that is known from the work of Balogh, Bertola, Lee, and McLaughlin~\cite{BBLM}. The first row of the solution is
 \begin{equation*}
 Y(z) =
 \begin{bmatrix}
 P_n(z) & \displaystyle \frac{1}{2\pi {\rm i}} \int_{\gamma} \frac{P_n(s) (s-a)^c {\rm e}^{-\overline{a} s}}{s^{n+c}} \frac{{\rm d}s}{s-z} \\
 * & *
 \end{bmatrix}\!, \qquad z \in \Cset \setminus \gamma.
 \end{equation*}

 \section{Type II multiple orthogonality}
 \label{section5}
 The planar orthogonality corresponding to~\eqref{eq:muW}--\eqref{eq:Wdef} is equivalent to the type~II multiple orthogonality. This has been established by Lee and Yang in~\cite{LY2}. We give a proof in the case of polynomial $W$, which is essentially the same as the one in~\cite{LY2}, except that the situation is more transparent due to the lack of branch cuts for~$W$ and~$\phi_k$.

 Type II multiple orthogonality means that there exists functions~$w_1, \dots, w_p$ and non-negative integers~$n_1, \dots, n_p$ with~$n = \sum_{k=1}^p n_k$ such that~$P_n$ is the unique polynomial of degree~$n$ satisfying
 \begin{equation}
 \label{eq:typeII}
 \frac{1}{2\pi {\rm i}} \oint_{\gamma} P_n(z) W(z) z^m w_k(z)\, {\rm d}z = 0
 \end{equation}
 for~$k=1, \dots, p$ and~$m = 0, \dots, n_k-1$, where~$\gamma$ is a closed contour around the origin.

 \begin{Theorem}
 \label{thm:typeII}
 Suppose~$W$ is a polynomial. Given~$n \in \Nset$, choose integers~$n_1, \dots, n_p$ such that
 \begin{equation}
 \label{eq:njdef}
 \sum_{j=1}^p n_j = n, \qquad
 \bigg\lfloor \frac{n}{p} \bigg\rfloor \leq n_j \leq \bigg\lceil \frac{n}{p}\bigg \rceil,
 \end{equation}
 and define~$w_k$, $k=1, \dots, p$, by
 \begin{equation} \label{eq:wkdef} w_k(z) = \int_0^{\bar{z} \times \infty} \prod_{j=1}^p (u-\overline{a_j})^{c_j + n_j-\delta_{k,j}} {\rm e}^{-uz} \, {\rm d}u.
 \end{equation}
 Then, the planar orthogonal polynomial~$P_n$ is the unique monic polynomial of degree~$n$ that satisfies~\eqref{eq:typeII}.
 \end{Theorem}

 Note that~\eqref{eq:njdef} implies~$|n_j-n_k| \leq 1$ for every~$j,k=1, \dots, p$. Euclid's division lemma yields
 \begin{equation*}
 n = ap + b
 \end{equation*}
 for some~$a,b \in \Nset \cup \{ 0\}$ such that~$0 \leq b < p$. And it is clear that~$a+1$ and~$a$ will appear in~$(n_1, \dots, n_p)$ exactly~$b$ and~$p-b$ times, respectively.

 \begin{proof}[Proof of Theorem \ref{thm:typeII}]
 Suppose that~$P_n$ is the degree~$n$ monic planar orthogonal polynomial. Take $k \in \{1, \dots, p\}$, and let~$m$ be an integer such that~$0 \leq m \leq n_k-1$. Carrying out~$m$ integration by parts in~\eqref{eq:wkdef}, we get{\samepage
 \begin{align}
 z^m w_k(z) & = (-1)^m \int_0^{\bar{z} \times \infty} \prod_{j=1}^p (u-\overline{a_j})^{c_j + n_j-\delta_{k,j}} \bigg[ \bigg(\frac{\rm d}{{\rm d}u}\bigg)^m {\rm e}^{-uz} \bigg] {\rm d}u \nonumber
 \\
 & = \int_0^{\bar{z} \times \infty} \bigg[ \bigg(\frac{\rm d}{{\rm d}u}\bigg)^m \prod_{j=1}^p (u-\overline{a_j})^{c_j + n_j-\delta_{k,j}} \bigg] {\rm e}^{-uz} \, {\rm d}u + \Pi_{k,m}(z),
 \label{eq:typeIIproof1}
 \end{align}
 where $\Pi_{k,m}$ is a polynomial that comes from the boundary terms at~$u=0$.}

 Observe that
 \begin{equation}
 \label{eq:typeIIproof2}
 \bigg(\frac{\rm d}{{\rm d}u}\bigg)^m \prod_{j=1}^p (u-\overline{a_j})^{c_j + n_j-\delta_{k,j}}
 \end{equation}
 is a polynomial in~$u$ of degree~$\leq c+n-1-m$ with a zero of order~$c_j+n_j - \delta_{k,j} - m$ at~$\overline{a_j}$ for every~$j=1, \dots, p$. From the definition~\eqref{eq:njdef} it follows that~$|n_k - n_j| \leq 1$, which implies that~$n_k \leq n_j + 1 - \delta_{k,j}$ for every~$j=1, \dots, p$. Using~$m \leq n_k-1$, we find that~\eqref{eq:typeIIproof2} has a zero at~$\overline{a_j}$ of order~$\geq c_j$. Therefore,~\eqref{eq:typeIIproof2} is divisible by~$W^*(u)$, and we can write
 \begin{equation}
 \label{eq:typeIIproof3}
 \bigg(\frac{\rm d}{{\rm d}u}\bigg)^m \prod_{j=1}^p (u-\overline{a_j})^{c_j + n_j-\delta_{k,j}} = W^*(u) Q_{k,m}(u),
 \end{equation}
 where $Q_{k,m}$ is a polynomial of degree
$\deg(Q_{k,m}) = n-1-m \leq n-1$
 and~$Q_{k,m}$ has a zero at~$\overline{a_j}$ of order~$n_j - \delta_{k,j}-m$ for every~$j=1, \dots, p$.

 In view of~\eqref{eq:typeIIproof1} and~\eqref{eq:typeIIproof3}, we then have
 \begin{equation*}
 \begin{aligned}
 \frac{1}{2\pi {\rm i}} \oint_{\gamma} P_n(z) W(z) z^m w_k(z)\, {\rm d}z ={} &\frac{1}{2\pi {\rm i}} \oint_{\gamma} P_n(z) W(z) \Bigg[ \int_0^{\bar{z} \times \infty} W^*(u) Q_{k,m}(u) {\rm e}^{-uz} {\rm d}u\Bigg] {\rm d}z
 \\
 &{}+\frac{1}{2\pi {\rm i}} \oint_{\gamma} P_n(z) W(z)\Pi_{k,m}(z)\, {\rm d}z.
 \end{aligned}
 \end{equation*}
 The second term vanishes because of Cauchy's theorem, and the remaining term vanishes because of part $(b)$ in Theorem~\ref{thm:theo1} and the fact that~$\deg(Q_{k,m}) \leq n-1$. Hence, \eqref{eq:typeII} holds.

 Conversely, suppose that~$P_n$ satisfies~\eqref{eq:typeII} with~$w_k$ and~$n_k$ as in the statement of the theorem. Then, by~\eqref{eq:typeIIproof1} and~\eqref{eq:typeIIproof3}, we get
 \begin{equation*}
 \frac{1}{2\pi {\rm i}} \oint_{\gamma} P_n(z) W(z) \int_0^{\bar{z} \times \infty} W^*(u) Q_{k,m}(u) {\rm e}^{-uz}\, {\rm d}u\, {\rm d}z = 0,
 \end{equation*}
 for~$k=1, \dots, p$ and~$m = 0, \dots, n_k-1$. Again, there is no contribution from the polynomial~$\Pi_{k,m}$. This leads directly to~\eqref{eq:Pnequiv2}, provided that the~$Q_{k,m}$ are a basis of the vector space of polynomials of degree~$\leq n-1$. Then, Theorem~\ref{thm:theo1} tells us that~$P_n$ is the degree~$n$ planar orthogonal polynomial.

 The polynomials~$Q_{k,m}$ for $k=1, \dots, p$ and~$m=0, \dots, n_k-1$ have degrees~$\leq n-1$, and because of~\eqref{eq:njdef} there are~$n$ of them. Thus, it suffices to prove that the~$Q_{k,m}$ are linearly independent. Suppose that~$\beta_{k,m}$ are complex numbers such that
 \begin{equation}
 \label{eq:typeIIproof4}
 \sum_{k=1}^p \sum_{m=0}^{n_k-1} \beta_{k,m} Q_{k,m} = 0.
 \end{equation}
 We already noted in the first part of the proof that~$Q_{k,m}$ has a zero at~$\overline{a_j}$ of exact order $n_j-\delta_{k,j}-m$. Hence, $Q_{k,m}(\overline{a_j}) \neq 0$ if and only if~$n_j-\delta_{k,j}-m =0$.

 From~\eqref{eq:njdef}, we have that~$|n_j-n_k| \leq 1$, and since~$m \leq n_k-1$, it is then easy to see that~$Q_{k,m}(\overline{a_j}) \neq 0$ if and only if~$m = n_{k}-1$, and either~$k=j$, or~$k \neq j$ and $n_k = n_j+1$. Thus, by evaluating~\eqref{eq:typeIIproof4} at~$\overline{a_j}$, we obtain
 \begin{equation}
 \label{eq:typeIIproof5}
 \beta_{j,n_j-1} Q_{j,n_j-1}(\overline{a_j}) + \sum_{\substack{k=1\\ n_k = n_j+1}}^p \beta_{k,n_k-1} Q_{k,n_k-1}(\overline{a_j}) = 0.
 \end{equation}

 Suppose~$n_{j_0} = \big\lceil \frac{n}{p} \big\rceil$. Then, there are no indices~$k$ with~$n_k = n_{j_0}+1$, and~\eqref{eq:typeIIproof5} implies that~$\beta_{{j_0},n_{j_0}-1} = 0$ since~$Q_{j_0,n_{j_0}-1}(\overline{a_{j_0}}) \neq 0$. Suppose~$n_{j_0} = \big\lfloor \frac{n}{p} \big\rfloor$. Then, every~$k$ with~$n_k = n_{j_0}+1$ satisfies~$n_k = \big\lceil \frac{n}{p} \big\rceil$, and we just proved that~$\beta_{k,n_k-1} = 0$ for such~$k$. Thus, by using~\eqref{eq:typeIIproof5} again we obtain that~$\beta_{{j_0},n_{j_0}-1} = 0$. Since~$j_0$ can be arbitrary, we conclude that~$\beta_{{j},n_{j}-1} = 0$ for all~$j=1, \dots, p$.

 The formula~\eqref{eq:typeIIproof5} reduces to
 \begin{equation}
 \label{eq:typeIIproof6}
 \sum_{k=1}^p \sum_{m=0}^{n_k-2} \beta_{k,m} Q_{k,m} = 0.
 \end{equation}
 We continue by looking at the remaining polynomials~$Q_{k,m}$ in~\eqref{eq:typeIIproof6} and observe that the only ones with~$\frac{\rm d}{{\rm d}z}Q_{k,m}\big|_{z=\overline{a_j}} \neq 0$ are those with~$m = n_k-2$ and either~$k=j$, or~$k\neq j$ and~$n_k=n_j+1$. Arguing as before we find that~$\beta_{j,n_j-2} = 0$ for every~$j=1, \dots, p$. Continuing in this way by considering the higher order derivatives, we ultimately find that~$\beta_{k,m} = 0$ for all~$k=1, \dots, p$ and $m=0, \dots, n_k-1$, which shows that the polynomials~$Q_{k,m}$ are indeed linearly independent.
 \end{proof}

 The connection between Riemann--Hilbert problems and type II multiple orthogonality is well known, see Van Assche, Geronimo, and Kuijlaars~\cite{VAGK}. Therefore, we arrive at the following Riemann--Hilbert problem of size~$(p+1) \times (p+1)$, corresponding to Theorem~\ref{thm:typeII} (first appeared in Lee and Yang~\cite{LY2}).

 \begin{rhproblem}[type II multiple orthogonality]
 Let $w_1, \dots, w_p$ be given by~formula~\eqref{eq:wkdef}. Find~$Y \colon \Cset \setminus \gamma \to \Cset^{(p+1) \times (p+1)}$ satisfying
 \begin{enumerate}
 \item[$(1)$] $Y(z)$ is analytic for~$z \in \Cset \setminus \gamma$;

 \item[$(2)$] $Y_+(z) = Y_-(z) J_Y(z)$ for $z \in \gamma$, where
$
 J_Y(z) =
 \begin{bmatrix}
 1 & w_1(z)&\cdots& w_{p}(z)\\
 0 & 1& \cdots & 0\\
 \vdots & \vdots&\ddots&\vdots\\
 0&0&\cdots& 1
 \end{bmatrix}
$;

 \item[$(3)$] $Y(z) = \big(I+O\big(\frac{1}{z}\big)\big)
 \begin{bmatrix}
 z^n&0&\cdots&0\\
 0&z^{-n_1}&\cdots& 0\\
 \vdots&\vdots& \ddots&\vdots\\
 0&0&\cdots&z^{-n_p}
 \end{bmatrix}
 $ as $z \to \infty$.
 \end{enumerate}
 \end{rhproblem}

 The unique solution has $P_n$ in its first row
 \begin{equation} \label{eq:YtypeII}
 Y(z)=
 \begin{bmatrix}
 P_n(z) & \displaystyle \frac{1}{2 \pi {\rm i}} \oint_\gamma P_n(s) w_{1}(s) \frac{{\rm d}s}{s-z} &\cdots & \displaystyle \frac{1}{2 \pi {\rm i}} \oint_\gamma P_n(s) w_{p}(s) \frac{{\rm d}s}{s-z}\\
 * & *&\cdots&*\\
 \vdots & \vdots& \ddots& \vdots\\
 * & *&\cdots&*
 \end{bmatrix}
 \end{equation}
 for $z \in \Cset \setminus \gamma$.

 \section{Conclusion}
 \label{sec:conc}
 It is known that the planar orthogonal polynomials orthogonal with respect to the measure~\eqref{eq:muW} are multiple orthogonal polynomials of type~II (see~\cite{LY2}). Assuming $W$ is a polynomial weight, we have shown that the same planar orthogonal polynomials are also multiple orthogonal polynomials of type~I. It is remarkable that the planar orthogonality manifests in two different ways at once. We are not aware of any other examples of such a phenomenon except for the case $p=1$, where both situations reduce to the usual orthogonality.

 Generally speaking, multiple orthogonality of each type,~I and~II, is characterized by a~Rie\-mann--Hilbert problem. Moreover, there is a~canonical correspondence between such Riemann--Hilbert problems. Indeed, if~$Y$ is given by~\eqref{eq:YtypeII} then not only does it solve a Riemann--Hilbert problem for multiple orthogonal polynomials of type~II, but also its inverse transpose~$Y^{-t}$ contains multiple orthogonal polynomials of type~I in each of its rows. In the course of the paper, we have established several different Riemann--Hilbert problems connected with orthogonality of type~I. However, if~$p \geq 2$, then the orthogonal polynomial~$P_n$ enters the inverse transpose of~\eqref{eq:YtypeII} only as part of a bigger algebraic expression and not by itself as it should in the case of the canonical correspondence; therefore, neither of our Riemann--Hilbert problems associated to multiple orthogonality of type~I are related to the type II problem in a canonical way.

 A major interest in stating the Riemann--Hilbert problems is to use them for the asymptotic analysis. We do not address this topic in the present paper, however we mention that Lee and Yang~\cite{LY3} used the Riemann--Hilbert problem corresponding to multiple orthogonality of type~II for asymptotic analysis in the situation where the~$c_j$ are fixed and~$n \to \infty$. It would be very interesting to deal with the case of varying weights, namely, when the $c_j$ depend on $n$ and tend to infinity at a rate proportional to~$n$. For~$p=1$ this was accomplished in~\cite{BBLM}, and we hope that one of the Riemann--Hilbert problems in our paper can be useful for the case~$p \geq 2$.

 \subsection*{Acknowledgments}
S.B.\ is supported by FWO Senior Postdoc Fellowship, project 12K1823N. A.B.J.K.\ was supported by the long term structural funding ``Methusalem grant of the Flemish Government'', and by FWO Flanders projects EOS 30889451 and G.0910.20. I.P.\ was supported by FWO Flanders project G.0910.20.

\pdfbookmark[1]{References}{ref}
\LastPageEnding

\end{document}